\documentclass{article}
\usepackage{amsmath}
\usepackage{amsfonts}
\usepackage{amsthm}
\usepackage{amssymb}
\usepackage{color}

\newtheorem{theorem}{Theorem}[section]

\newtheorem{observation}[theorem]{Observation}

\theoremstyle{definition}
\newtheorem{definition}[theorem]{Definition}
\theoremstyle{remark}
\newtheorem{remark}[theorem]{Remark}

\def\f2{\mathbb{F}_2}

\def\dist{\hskip0.02cm{\rm dist}\hskip0.01cm}
\def\lip{\hskip0.02cm{\rm Lip}\hskip0.01cm}

\newcommand{\lin}{{\rm lin}\hskip0.02cm}

\begin{document}
\title{\LARGE{Embeddability of locally finite metric spaces into
Banach spaces is finitely determined }}

\author{M.\,I.~Ostrovskii}

\date{\today}
\maketitle

\noindent{\bf Abstract.} The main purpose of the paper is to prove
the following results:
\begin{itemize}
\item Let $A$ be a locally finite metric space whose finite
subsets admit uniformly bilipschitz embeddings into a Banach space
$X$. Then $A$ admits a bilipschitz embedding into $X$.

\item Let $A$ be a locally finite metric space whose finite
subsets admit uniformly coarse embeddings into a Banach space $X$.
Then $A$ admits a coarse embedding into $X$.

\end{itemize}

These results generalize previously known results of the same type
due to Brown--Guentner (2005), Baudier (2007), Baudier--Lancien
(2008), and the author (2006, 2009).\medskip

One of the main steps in the proof is: each locally finite subset
of an ultraproduct $X^\mathcal{U}$ admits a bilipschitz embedding
into $X$. We explain how this result can be used to prove
analogues of the main results for other classes of embeddings.
\medskip


\begin{large}

\section{Introduction}

First we introduce necessary definitions:

\begin{definition} {\rm A metric space $(A,d_A)$ is called {\it discrete} if there exists a
constant $\delta > 0$ such that $\forall u,v\in A\quad
d_A(u,v)\ge\delta$. A discrete metric space $A$ is called {\it
locally finite} if for every $u\in A$ and every $r>0$ the set
$\{a\in A:~d_A(u,a)\le r\}$ is finite.\smallskip

Let $C<\infty$. A map $f: (A,d_A)\to (Y,d_Y)$ between two metric
spaces is called $C$-{\it Lipschitz}\label{lip} if \[\forall
u,v\in A\quad d_Y(f(u),f(v))\le Cd_A(u,v).\] A map $f$ is called
{\it Lipschitz} if it is $C$-Lipschitz for some $C<\infty$. For a
Lipschitz map $f$ we define its {\it Lipschitz
constant}\label{lipC} by
\[\lip f:=\sup_{d_A(u,v)\ne 0}\frac{d_Y(f(u),
f(v))}{d_A(u,v)}.\label{lipCE}\]

A map $f:A\to Y$ is called a {\it $C$-bilipschitz embedding} if
there exists $r>0$ such that
\begin{equation}\label{E:MapDist}\forall u,v\in A\quad rd_A(u,v)\le
d_Y(f(u),f(v))\le rCd_A(u,v).\end{equation} A {\it bilipschitz
embedding} is an embedding which is $C$-bilipschitz for some
$C<\infty$. The smallest constant $C$ for which there exist $r>0$
such that \eqref{E:MapDist} is satisfied is called the {\it
distortion} of $f$. A sequence of embeddings is called {\it
uniformly bilipschitz} if they have uniformly bounded
distortions.\smallskip

A map $f:(X,d_X)\to (Y,d_Y)$ between two metric spaces is called a
{\it coarse embedding} if there exist non-decreasing functions
$\rho_1,\rho_2:[0,\infty)\to[0,\infty)$ (observe that this
condition implies that $\rho_2$ has finite values) such that
$\lim_{t\to\infty}\rho_1(t)=\infty$ and
\begin{equation}\label{E:coarse}\forall u,v\in X~ \rho_1(d_X(u,v))\le
d_Y(f(u),f(v))\le\rho_2(d_X(u,v)).\end{equation} A sequence of
embeddings is called {\it uniformly coarse} if all of them satisfy
\eqref{E:coarse} with the same $\rho_1$ and $\rho_2$.}
\end{definition}

The main purpose of this paper is to prove the following two
results:

\begin{theorem}\label{T:bilip} Let $A$ be a locally finite metric space whose
finite subsets admit uniformly bilipschitz embeddings into a
Banach space $X$. Then $A$ admits a bilipschitz embedding into
$X$.
\end{theorem}

\begin{theorem}\label{T:UnCoarse} Let $A$ be a locally finite metric space whose
finite subsets admit uniformly coarse embeddings into a Banach
space $X$. Then $A$ admits a coarse embedding into $X$.
\end{theorem}

\begin{remark} It is worth mentioning that our argument implies that a
similar result holds for any class $\mathcal{E}$ of embeddings
provided that

\begin{itemize}

\item[(a)] There is a notion of being {\it uniformly in}
$\mathcal{E}$ for a collection of maps of finite metric spaces
into a Banach space.

\item[(b)] The notion in (a) is such that if all finite subspaces
of a metric space $A$ admit uniformly-in-$\mathcal{E}$ embeddings
into a Banach space $X$, then there is an embedding of the class
$\mathcal{E}$ of $A$ into an ultraproduct $X^{\mathcal{U}}$ where
$\mathcal{U}$ is a non-trivial ultrafilter (see the construction
below).

\item[(c)] The image of a locally finite metric space under an
embedding of the class $\mathcal{E}$ is locally finite.

\item[(d)] A composition of an embedding of the class
$\mathcal{E}$ and a bilipschitz embedding is in $\mathcal{E}$.

\end{itemize}
\end{remark}

Theorems \ref{T:bilip} and \ref{T:UnCoarse} generalize previous
results of the same type obtained in
\cite{Bau07,BL08,BG05,Ost06a,Ost06b,Ost09}, see section
\ref{S:previous}.\medskip

Our proof uses (a) The way of pasting embeddings of ``pieces''
suggested in \cite{BL08}. (b) Approaches to selection of basic
subsequences developed in \cite{KP65}. (c) Some basic ultraproduct
techniques going back to \cite{DK72}.

\section{Proof in the bilipschitz case}

\begin{proof}[Proof of Theorem \ref{T:bilip}] We pick a point $O$
in $A$ and let $A_i=\{a\in A:~d_A(O,a)\le 2^i\}$. By the
assumption there are uniformly bilipschitz maps $f_i:A_i\to X$. We
may and shall assume that $f_i(O)=0$ and that there is a constant
$1\le C<\infty$ such that
\[\forall u,v\in A_i\quad d_A(u,v)\le||f_i(u)-f_i(v)||\le Cd_A(u,v).\]

We are going to use some basic facts about ultraproducts of Banach
spaces introduced in \cite{DK72}. We refer to \cite[Chapter
8]{DJT95} for background on this matter.\medskip

Let $\mathcal{U}$ be a nontrivial ultrafilter on $\mathbb{N}$. The
maps $\{f_i\}_{i=1}^\infty$ induce a map $f:A\to X^{\mathcal{U}}$
defined by $f(u)=\{\widetilde f_i(u)\}_{i=1}^\infty$, where
\[\widetilde f_i(u)=\begin{cases} f_i(u) &\hbox{ if } u\in A_i\\
0 & \hbox{ if }u\notin A_i.\end{cases}\]

The definition of an ultraproduct immediately implies that $f:A\to
X^{\mathcal{U}}$ is a bilipschitz embedding. Let $N=f(A)$. Since
the composition of two bilipschitz embeddings is a bilipschitz
embedding, it suffices to find a bilipschitz embedding of $N$
(with the metric induced from $X^\mathcal{U}$) into $X$.
\medskip

\noindent{\bf Note:} This passage from $A$ to its image in
$X^\mathcal{U}$ is not essential for the proof of Theorem
\ref{T:bilip}, it just simplifies some formulas in our proof. A
similar step is more essential for other classes of embeddings.

\begin{remark}\label{R:MAIN} We would like to emphasize that the
rest of the proof of Theorem \ref{T:bilip} consists in
establishing the fact that a locally finite metric subspace of
$X^{\mathcal{U}}$ admits a bilipschitz embedding into $X$.
\end{remark}

\begin{observation}\label{O:SimpleCases} If $X$ is finite-dimensional, then
$X^\mathcal{U}$ is of the same dimension {\rm (see
\cite[Proposition 8.4]{DJT95})}, and the proof if completed.

If $X=L_p(0,1)$ for some $p\in[1,\infty]$, then each separable
subspace of $X^\mathcal{U}$ is isometric to a subspace of $X$
{\rm(see \cite[Theorem 8.7]{DJT95} and references in
\cite[p.~169]{Ost09})}, so the proof is completed in this case,
too.
\end{observation}

In this connection in the rest of the proof we assume that $X$ is
infinite-dimen\-sio\-nal.\medskip

Let $N_i=\{u\in N:~||u||_{X^\mathcal{U}}\le 2^i\}$. It is clear
that $N_i$ are finite sets. Using the same argument as in the
proof of finite representability of $X^\mathcal{U}$ in $X$ (see
\cite[Theorem 8.13]{DJT95}) we get that there exist maps
$s_i:N_i\to X$ such that $s_i(0)=0$ and
\begin{equation}\label{E:*}
\forall u,v\in
N_i\quad||u-v||\le||s_i(u)-s_i(v)||\le\left(1+\frac1i\right)||u-v||.
\end{equation}

Since the sets $N_i$ form an increasing sequence, any subsequence
$\{s_{n_i}\}_{i=1}^\infty$ of $\{s_i\}_{i=1}^\infty$ maps
$\{N_i\}_{i=1}^\infty$ into $X$ and satisfies \eqref{E:*}. We are
going to construct a bilipschitz embedding of $N$ into $X$ using
such subsequences.\medskip

\noindent{\bf Note:} We are going to pass to a subsequence in
$\{s_i\}_{i=1}^\infty$ several times. Each time we keep the
notation $\{s_i\}_{i=1}^\infty$ for the subsequence.\medskip

Recall that a subspace $M\subset X^*$ is called {\it $1$-norming}
if
\[\forall x\in X~\sup\{|f(x)|:~f\in M,~||f||\le 1\}=||x||.\]

It is clear that we may assume that $X$ is separable (replacing it
by the closure of the linear span of $\bigcup_{i=1}^\infty
s_i(N_i)$, if necessary).
\medskip

For a separable Banach space $X$ there exists a separable
$1$-norming subspace $M\subset X^*$. It can be constructed as
follows: Let $\{x_i\}_{i=1}^\infty$ be a dense sequence in the
unit sphere $S_X=\{x\in X:~||x||=1\}$. Let $f_i\in S_{X^*}$ be
such that $f_i(x_i)=1$. It is easy to check that the closed linear
span $M$ of the sequence $\{f_i\}_{i=1}^\infty$ is $1$-norming.
\medskip

Let $M\subset X^*$ be a separable $1$-norming subspace. Then the
natural embedding of $X$ into $M^*$ is an isometry. We identify
$X$ with its image under this embedding. Since $M$ is separable,
there is a subsequence in $\{s_i\}$ such that the sequence
$\{s_i(a)\}_{i=k}^\infty$ is convergent in the weak$^*$ topology
of $M^*$ for each $a\in N_k$. We denote the weak$^*$ limit of this
sequence by $m(a)$.
\medskip

We need to select further subsequences of $\{s_i\}$. We do this in
the following two steps.\medskip

\noindent{\bf Step 1.} If for some $a, b\in N_i$ and some $j\ge i$
the vector $(s_j(a)-s_j(b))-(m(a)-m(b))$ is nonzero, we find and
fix $f=f_{j,a,b}\in S_{M}$ such that
\begin{equation}\label{E:Weak1}
f((s_j(a)-s_j(b))-(m(a)-m(b)))\ge\frac{99}{100}||(s_j(a)-s_j(b))-(m(a)-m(b))||.\end{equation}
In such a case we assume that for all $k>j$ the condition
\begin{equation}\label{E:Weak2}|f((s_k(a)-s_k(b))-(m(a)-m(b)))|\le\frac1{1000}||a-b||
\end{equation}
holds. This goal can be achieved because there are finitely many
$a, b\in N_i$ and because $s_k(a)$ converges to $m(a)$ in the
weak$^*$ topology of $M^*$. \medskip

\noindent{\bf Step 2.} If $m(a)\ne m(b)$ for $a,b\in N_i$, we find
and fix $f=f_{a,b}\in S_M$ such that
\begin{equation}\label{E:mab}
f(m(a)-m(b))\ge\frac{99}{100}||m(a)-m(b)||\end{equation} and
select a subsequence satisfying
\begin{equation}\label{E:Weak3}|f((s_k(a)-s_k(b))-(m(a)-m(b)))|\le\frac1{100}||m(a)-m(b)||\end{equation}
for $k\ge i$. This can be achieved because $N_i$ is finite and
$s_k(a)$ converges to $m(a)$ in the weak$^*$ topology of
$M^*$.\medskip

We introduce a map $\varphi:A\to X$ by
\[\varphi(a)=\frac{2^i-||a||}{2^{i-1}}s_i(a)+
\frac{||a||-2^{i-1}}{2^{i-1}}s_{i+1}(a)\] if $2^{i-1}\le||a||\le
2^i$. One can easily check that the map is well-defined for
$||a||=2^i$.
\medskip

We start by considering the case where $X$ is isomorphic to its
hyperplane, and therefore $X$ is isomorphic to
$X\oplus\mathbb{R}$. In this case we show that the embedding
$\widetilde\varphi:A\to X\oplus\mathbb{R}$ given by
$\widetilde\varphi(a)=(\varphi(a),||a||)$ is a bilipschitz
embedding.\medskip

The proof in the case when $X$ is not isomorphic to its hyperplane
is completed in section \ref{S:WithoutLine}. (We know, by results
of \cite{GM93}, that spaces which are not isomorphic to their
hyperplanes exist.)
\medskip

Now we estimate the Lipschitz constants of $\widetilde\varphi$ and
$(\widetilde\varphi)^{-1}$. We consider three cases.

\subsection{Case 1: $2^{i-1}\le||b||\le||a||\le 2^i$}

In this case we have
\begin{equation}\label{E:case1}\begin{split}
\varphi (a)-\varphi(b)&=\frac{2^i-||a||}{2^{i-1}}s_i(a)+
\frac{||a||-2^{i-1}}{2^{i-1}}s_{i+1}(a)\\&\qquad-\frac{2^i-||b||}{2^{i-1}}s_i(b)-
\frac{||b||-2^{i-1}}{2^{i-1}}s_{i+1}(b)\\&=
\frac{2^i-||a||}{2^{i-1}}(s_i(a)-s_i(b))+\frac{||a||-2^{i-1}}{2^{i-1}}(s_{i+1}(a)-s_{i+1}(b))\\&
\qquad+\frac{||b||-||a||}{2^{i-1}}s_i(b)+
\frac{||a||-||b||}{2^{i-1}}s_{i+1}(b).
\end{split}\end{equation}

Using \eqref{E:*} we get
\[\begin{split}||\varphi (a)-\varphi(b)||&\le
\frac{2^i-||a||}{2^{i-1}}\,2||a-b||+\frac{||a||-2^{i-1}}{2^{i-1}}\,2||a-b||\\&
\qquad+4|\,||b||-||a||\,|+ 4|\,||a||-||b||\,|,\end{split}\] and
the fact that $\widetilde\varphi$ is a Lipschitz map is
immediate.\medskip

Now we estimate the Lipschitz constant of
$(\widetilde\varphi)^{-1}$. So we need to estimate
$||\varphi(a)-\varphi(b)||$ from below. Observe that

\begin{equation}\label{E:repres}\begin{split}
\varphi (a)-\varphi(b)&=m(a)-m(b)\\
&\quad+
\frac{2^i-||a||}{2^{i-1}}(s_i(a)-s_i(b)-(m(a)-m(b)))\\
&\quad+\frac{||a||-2^{i-1}}{2^{i-1}}(s_{i+1}(a)-s_{i+1}(b)-(m(a)-m(b)))\\&
\quad+(||b||-||a||)\cdot\frac{s_i(b)}{2^{i-1}}+
(||a||-||b||)\frac{s_{i+1}(b)}{2^{i-1}}.
\end{split}\end{equation}

First we consider the case when
$||m(a)-m(b)||\ge\frac1{100}||a-b||$. Let $f_{a,b}$ be the
corresponding functional (see Step 2 above). We have

\begin{equation}\label{E:below1}\begin{split}
||\varphi(a)-\varphi(b)||&\ge f_{a,b}(\varphi(a)-\varphi(b))\\
&= f_{a,b}(m(a)-m(b))\\
&\quad+f_{a,b}\left(
\frac{2^i-||a||}{2^{i-1}}(s_i(a)-s_i(b)-(m(a)-m(b)))\right)\\
&\quad+f_{a,b}\left(\frac{||a||-2^{i-1}}{2^{i-1}}(s_{i+1}(a)-s_{i+1}(b)-(m(a)-m(b)))\right)\\&
\quad+(||b||-||a||)\cdot\frac{f_{a,b}(s_i(b))}{2^{i-1}}+
(||a||-||b||)\frac{f_{a,b}(s_{i+1}(b))}{2^{i-1}}\\
&\stackrel{\eqref{E:mab}~\&~\eqref{E:Weak3}}{\ge}\frac{99}{100}||m(a)-m(b)||-\frac1{100}||m(a)-m(b)||\\
&\qquad-4|\,||b||-||a||\,|- 4|\,||a||-||b||\,|\\
&\ge\frac{98}{10000}||a-b||-8(||a||-||b||).
\end{split}\end{equation}

In the case when $||a||-||b||<\frac1{1000}||a-b||$, we get an
estimate for the $\lip(\varphi^{-1})$ (and thus for
$\lip((\widetilde\varphi)^{-1})$\,) from above.
\medskip

The estimate for $\lip((\widetilde\varphi)^{-1})$ in the case when
$||a||-||b||\ge\frac1{1000}||a-b||$ is immediate, we just recall
that
\begin{equation}\label{E:estTilde}
\widetilde\varphi(a)-\widetilde\varphi(b)=(\varphi(a)-\varphi(b))\oplus(||a||-||b||).\end{equation}

Hence, to finish the estimate for $\lip((\widetilde\varphi)^{-1})$
in Case 1 it remains to consider the case when
$||m(a)-m(b)||<\frac1{100}||a-b||$. In this case we consider two
subcases:

\begin{equation}\label{E:subcase1}
\frac{2^i-||a||}{2^{i-1}}||s_i(a)-s_i(b)-(m(a)-m(b))||\ge\frac1{10}||a-b||
\end{equation}

\begin{equation}\label{E:subcase2}
\frac{2^i-||a||}{2^{i-1}}||s_i(a)-s_i(b)-(m(a)-m(b))||<\frac1{10}||a-b||
\end{equation}

We start with subcase \eqref{E:subcase1}. Let $f_{i,a,b}$ be the
functional found in Step 1. We get

\begin{equation}\label{E:below2}\begin{split}
||\varphi(a)-\varphi(b)||&\ge f_{i,a,b}(\varphi(a)-\varphi(b))\\
&\ge f_{i,a,b}(m(a)-m(b))\\
&\quad+f_{i,a,b}\left(
\frac{2^i-||a||}{2^{i-1}}(s_i(a)-s_i(b)-(m(a)-m(b)))\right)\\
&\quad+f_{i,a,b}\left(\frac{||a||-2^{i-1}}{2^{i-1}}(s_{i+1}(a)-s_{i+1}(b)-(m(a)-m(b)))\right)\\&
\quad+(||b||-||a||)\cdot\frac{f_{i,a,b}(s_i(b))}{2^{i-1}}+
(||a||-||b||)\frac{f_{i,a,b}(s_{i+1}(b))}{2^{i-1}}\\
&\stackrel{\eqref{E:Weak1}, \eqref{E:Weak2},~\&~\eqref{E:subcase1}}{>}\frac{99}{1000}||a-b||-\frac1{1000}||a-b||-\frac1{100}||a-b||\\
&\qquad\qquad-4|\,||b||-||a||\,|- 4|\,||a||-||b||\,|\\
&=\frac{88}{1000}||a-b||-8(||a||-||b||).
\end{split}\end{equation}

Now, as above, we consider the case when
$||a||-||b||<\frac1{1000}||a-b||$ separately, and complete the
argument in the same way as above.
\medskip

We turn to the subcase \eqref{E:subcase2}. Recall (see
\eqref{E:*}) that $||s_i(a)-s_i(b)||\ge ||a-b||$. Combining this
with \eqref{E:subcase2} and with the inequality
$||m(a)-m(b)||<\frac1{100}||a-b||$, we get
$||s_{i}(a)-s_{i}(b)-(m(a)-m(b))||\ge\frac{99}{100}||a-b||$ and
$\frac{2^i-||a||}{2^{i-1}}<\frac{10}{99}$. (In the same way we get
the inequality
$||s_{i+1}(a)-s_{i+1}(b)-(m(a)-m(b))||\ge\frac{99}{100}||a-b||$
which we use below.) Therefore
$\frac{||a||-2^{i-1}}{2^{i-1}}>\frac{89}{99}$. Applying the
triangle inequality we get
\begin{equation}\label{E:below3}\begin{split}
||\varphi
(a)-\varphi(b)||&\ge\left\|\frac{||a||-2^{i-1}}{2^{i-1}}(s_{i+1}(a)-s_{i+1}(b)-(m(a)-m(b)))\right\|
\\&\quad-
\left\|\frac{2^i-||a||}{2^{i-1}}(s_i(a)-s_i(b)-(m(a)-m(b)))\right\|\\
&\quad-||m(a)-m(b)||-8(||a||-||b||)\\
&>\frac{89}{99}\cdot\frac{99}{100}\,||a-b||-\frac1{10}||a-b||-\frac1{100}||a-b||-8(||a||-||b||)\\
&=\frac{78}{100}||a-b||-8(||a||-||b||).
\end{split}\end{equation}

Now, as it was done already twice, we consider the case when
$||a||-||b||<\frac1{100}||a-b||$ separately, and complete the
argument in the same way as above. This completes the argument in
Case 1.

\subsection{Case 2: $2^{i-1}\le||b||\le2^i\le||a||\le2^{i+1}$}

We have \[\label{E:case2}\begin{split} \varphi(a)-\varphi(b)
&=-\frac{2^i-||b||}{2^{i-1}}s_i(b)\\
&\quad+\frac{2^{i+1}-||a||}{2^{i}}s_{i+1}(a)-\frac{||b||-2^{i-1}}{2^{i-1}}s_{i+1}(b)\\
&\quad+\frac{||a||-2^{i}}{2^{i}}s_{i+2}(a)
\end{split}\]

{\bf Estimate from above:} The first and the last terms have norms
$\le4(||a||-||b||)$. The norm of the two remaining terms can be
estimated as follows:
\begin{equation}\label{E:Case2middle}\begin{split}
&\left\|\frac{2^{i+1}-||a||}{2^{i}}s_{i+1}(a)-\frac{||b||-2^{i-1}}{2^{i-1}}s_{i+1}(b)\right\|
\\&\qquad=\left\|\frac{2^i-(||a||-2^i)}{2^i}s_{i+1}(a)+
\frac{(2^i-||b||)-2^{i-1}}{2^{i-1}}s_{i+1}(b) \right\|\\
&\qquad=\left\|(s_{i+1}(a)-s_{i+1}(b))-\frac{(||a||-2^i)}{2^i}s_{i+1}(a)+\frac{(2^i-||b||)}{2^{i-1}}s_{i+1}(b)\right\|
\\ &\qquad\le||a-b||+4(||a||-2^i)+4(2^i-||b||)\le
5||a-b||.
\end{split}\end{equation}
\medskip

{\bf Now we turn to estimates from below.} Rewriting and
estimating some of the terms as in \eqref{E:Case2middle} we get
\begin{equation}\label{E:belowCase2}\begin{split}
||&\varphi(a)-\varphi(b)||\\
&\ge\left\|(s_{i+1}(a)-s_{i+1}(b))-\frac{(||a||-2^i)}{2^i}s_{i+1}(a)+\frac{(2^i-||b||)}{2^{i-1}}s_{i+1}(b)\right\|\\
&-\frac{2^i-||b||}{2^{i-1}}||s_i(b)||-\frac{||a||-2^{i}}{2^{i}}||s_{i+2}(a)||\\
&\ge||s_{i+1}(a)-s_{i+1}(b)||-12(||a||-||b||)\ge||a-b||-12(||a||-||b||),
\end{split}\end{equation}
where in the last line we used \eqref{E:*}. We complete the proof
in this case as three times before. If
$||a||-||b||<\frac1{20}||a-b||$, we get an estimate from
\eqref{E:belowCase2}. Otherwise we use \eqref{E:estTilde}.

\subsection{Case 3: $2^{i-1}\le||b||\le 2^i<2^{k-1}\le||a||\le 2^k$}

In this case we have
\[3(2^k+2^i)\ge3(||a||+||b||)\ge||\widetilde\varphi(a)-\widetilde\varphi(b)||\ge||a||-||b||\ge
2^{k-1}-2^i.\]
Since
\[\frac{3(2^k+2^i)}{2^{k-1}-2^i}\le\frac{3\cdot 2^{k+1}}{2^{k-2}}\le 24\]
and
\[||a||+||b||\ge||a-b||\ge||a||-||b||,\]
it follows that $\widetilde\varphi$ is bilipschitz.

\subsection{Completion of the proof for spaces non-isomorphic to their hyperplanes}\label{S:WithoutLine}

We have proved Theorem \ref{T:bilip} in the case when $X$ is
isomorphic to its hyperplane. To prove Theorem \ref{T:bilip} in
the general case we find a Lipschitz map $\tau:\mathbb{R}_+\to X$
such that the map $\widehat\varphi:N\to X$ given by
$\widehat\varphi(a)=\tau(||a||)+\varphi(a)$ works just in the same
way as $\widetilde\varphi$. It is easy to see that for this to be
true we need the inequality
\[||\widehat\varphi(a)-\widehat\varphi(b)||\ge\alpha(||a||-||b||)\]
to hold for some $\alpha>0$. We rewrite this inequality as
\begin{equation}\label{E:Need}||\tau(||a||)-\tau(||b||)+(\varphi(a)-\varphi(b))||\ge\alpha(||a||-||b||)
\end{equation}
Let $T_i=\{\varphi(u):~u\in N,~||u||\le 3^{i+1}\}$. It is clear
that all these sets are finite. We construct inductively a
sequence $\{F_i\}_{i=1}^\infty$ of finite-dimensional subspaces of
$X$ and a sequence $\{p_i\}_{i=1}^\infty$ of vectors. We let
$F_1=\lin(T_1)$. Since $X$ is infinite-dimensional (see the
assumption made after Observation \ref{O:SimpleCases}), there is
$p_1\in S_X$ such that $\dist(p_1,F_1)=1$. Let
$F_2=\lin(T_2\cup\{p_1\})$ and $p_2\in S_X$ be such that
$\dist(p_2,F_2)=1$. Let $F_3=\lin(T_3\cup\{p_i\}_{i=1}^2)$, we
continue in an obvious way.\medskip

We introduce the map $\tau:\mathbb{R}_+\to X$ in the following
way:
\[\tau(t)=\begin{cases} tp_1 & \hbox{ if } 0\le t\le 3^1\\
3^1p_1+(t-3^1)p_2 & \hbox{ if } 3^1\le t\le 3^2\\
3^1p_1+(3^2-3^1)p_2+(t-3^2)p_3 & \hbox{ if } 3^2\le t\le 3^3\\
\dots & \dots\\
3^1p_1+(3^2-3^1)p_2\dots+(3^k-3^{k-1})p_k+(t-3^k)p_{k+1}& \hbox{ if }3^k\le t\le 3^{k+1}\\
\dots& \dots\\
\end{cases}
\]

Since $||p_i||=1$, the map $\tau$ is $1$-Lipschitz. It remains to
show that the inequality \eqref{E:Need} holds. We consider three
cases:

\begin{enumerate}

\item $3^i\le||b||\le ||a||\le 3^{i+1}$. The argument used in this
case can be used also in the case $0\le||b||\le||a||\le 3^1$.
Minor adjustments of the other cases are needed if $0\le||b||\le
3^1\le||a||$.

\item $3^{i-1}\le ||b||\le 3^i\le ||a||\le 3^{i+1}$

\item $3^{k-1}\le ||b||\le 3^k\le 3^i\le ||a||\le 3^{i+1}$, where
$k<i$.

\end{enumerate}

In the first case we have
\[
||\tau(||a||)-\tau(||b||)+(\varphi(a)-\varphi(b))||=
\left\|(||a||-||b||)p_{i+1}+(\varphi(a)-\varphi(b))\right\|\ge
||a||-||b||.
\]
The last inequality follows from
$\dist(p_{i+1},F_{i+1})=||p_{i+1}||$ and $\varphi(a),\varphi(b)\in
T_i$, therefore $\varphi(a)-\varphi(b)\in F_i\subset F_{i+1}$.

In the second case we consider two subcases:

\begin{equation}\label{E:Sub1}
||a||-3^i\ge\frac13(||a||-||b||).
\end{equation}

\begin{equation}\label{E:Sub2}
||a||-3^i<\frac13(||a||-||b||).
\end{equation}

In subcase \eqref{E:Sub1} we get
\[\begin{split}
||\tau(||a||)-\tau(||b||)&+(\varphi(a)-\varphi(b))||\\
&=
\left\|(||a||-3^i)p_{i+1}+(3^i-||b||)p_i+(\varphi(a)-\varphi(b))\right\|\\
&\ge ||a||-3^i\ge \frac13(||a||-||b||),\end{split}
\]
where we use \eqref{E:Sub1}, $\dist(p_{i+1},F_{i+1})=||p_{i+1}||$,
and $p_i,\varphi(a),\varphi(b)\in F_{i+1}$.
\medskip

In subcase \eqref{E:Sub2} we have

\[\begin{split}
||\tau(||a||)-\tau(||b||)&+(\varphi(a)-\varphi(b))||\\
&=
\left\|(||a||-3^i)p_{i+1}+(3^i-||b||)p_i+(\varphi(a)-\varphi(b))\right\|\\
&\ge
\left\|(3^i-||b||)p_i+(\varphi(a)-\varphi(b))\right\|-(||a||-3^i)\\
&\ge (3^i-||b||)-(||a||-3^i)\ge \frac13(||a||-||b||),\end{split}
\]
In this chain of inequalities we use the fact that
$\varphi(a),\varphi(b)\in T_i$, $\dist(p_i,F_i)=||p_i||$, in the
last line we use the inequality $||b||\le 3^i\le ||a||$ and
\eqref{E:Sub2}.
\medskip

Now we consider the third case. In this case we again consider
subcases \eqref{E:Sub1} and \eqref{E:Sub2}. In the first subcase
the argument is exactly as above. So we focus on the second
subcase. In this subcase we have
\[||a||-3^i<\frac13(3^i-3^{i-1}),\]
see the defining inequality for the third case.
\medskip

In this subcase we have
\[\tau(||a||)-\tau(||b||)=(||a||-3^i)p_{i+1}+(3^i-3^{i-1})p_i+r,\]
where $r$ is a vector contained in $F_i$. Thus

\[\begin{split}
||\tau(||a||)-\tau(||b||)&+(\varphi(a)-\varphi(b))||\\
&\ge
\left\|(3^i-3^{i-1})p_i+r+(\varphi(a)-\varphi(b))\right\|-||(||a||-3^i)p_{i+1}||\\
&\ge (3^i-3^{i-1})-\frac13\,(3^i-3^{i-1})=\frac43\cdot
3^{i-1}\ge\frac4{27}(||a||-||b||),\end{split}
\]
where we use the fact that $r$ and $\varphi(a)-\varphi(b)$ are in
$F_i$ and $\dist(p_i,F_i)=||p_i||$. This completes the proof of
\eqref{E:Need} and thus Theorem \ref{T:bilip}.
\end{proof}

\section{Proof in the coarse case}

Proof of Theorem \ref{T:bilip} contains almost everything we need
for the proof of Theorem \ref{T:UnCoarse}, we need just to modify
the beginning of the proof.

\begin{proof}[Proof of Theorem \ref{T:UnCoarse}] We pick a point $O$
in $A$ and let $A_i=\{a\in A:~d_A(O,a)\le 2^i\}$. By the
assumption there are uniformly coarse maps $f_i:A_i\to X$. We may
and shall assume that $f_i(O)=0$. Let $\mathcal{U}$ be a
nontrivial ultrafilter on $\mathbb{N}$. The maps
$\{f_i\}_{i=1}^\infty$ induce a map $f:A\to X^{\mathcal{U}}$
defined by $f(u)=\{\widetilde f_i(u)\}_{i=1}^\infty$, where
\[\widetilde f_i(u)=\begin{cases} f_i(u) &\hbox{ if } u\in A_i\\
0 & \hbox{ if }u\notin A_i.\end{cases}\]

The definition of an ultraproduct immediately implies that $f:A\to
X^{\mathcal{U}}$ is a coarse embedding. Let $N=f(A)$, it is easy
to check that $N$ with the metric induced from $X^\mathcal{U}$ is
a locally finite metric space. The argument of the proof of
Theorem \ref{T:bilip} shows that there is a bilipschitz embedding
of $N$ into $X$ (see Remark \ref{R:MAIN}). Since the composition
of a coarse and a bilipschitz embeddings is a coarse embedding,
the proof is completed.\end{proof}

\section{Relations with previous results}\label{S:previous}

The main result of the paper \cite{BL08} is

\begin{theorem}\label{T:BL08} If $X$ is a Banach space without cotype, then every
locally finite metric space admits a bilipschitz embedding into
$X$.
\end{theorem}

Theorem \ref{T:BL08} is an immediate consequence of Theorem
\ref{T:bilip} and the following results (used in \cite{BL08}):

\begin{itemize}

\item \cite[p.~161]{Fre10} (see also \cite[p.~385]{Mat02}): Each
$n$-element metric space is isometric to a subset of
$\ell_\infty^n$.

\item \cite{MP76}: The spaces $\{\ell_\infty^n\}_{n=1}^\infty$
admit uniformly bilipschitz embeddings into any Banach space $X$
without cotype.

\end{itemize}

The fact that Theorem \ref{T:BL08} implies the following result of
\cite{BG05} was observed already in \cite{BL08}.

\begin{theorem}[\cite{BG05}] Let $A$ be a metric space with bounded geometry. There exists a
sequence of positive real numbers $\{p_n\}$ and a coarse embedding
of $A$ into the $\ell_2$-direct sum of
$\{\ell_{p_n}\}_{n=1}^\infty$.
\end{theorem}

In \cite{Ost09} the following result was proved

\begin{theorem} Let $A$ be a locally finite metric
space which admits a bilipschitz embedding into a Hilbert space,
and let $X$ be an infinite-dimensional Banach space. Then there
exists a bilipschitz embedding $f:A\to X$.
\end{theorem}

In \cite{Ost06b} a coarse version of this result was proved. These
results follow immediately from Theorems \ref{T:bilip} and
\ref{T:UnCoarse} and the Dvoretzky theorem \cite{Dvo61}.
\medskip

Theorem \ref{T:bilip} can also be used to derive both of the main
results of \cite{Bau07} from the mentioned in \cite{Bau07} finite
versions of the results.

\end{large}

\begin{small}

\renewcommand{\refname}{
\end{small}

\noindent{\sc Department of Mathematics and Computer Science\\
St. John's University\\ 8000 Utopia Parkway, Queens, NY 11439,
USA}\\
e-mail: {\tt ostrovsm@stjohns.edu}

\end{document}